\documentclass[11pt,draft]{amsart}
\input epsf.tex   

\usepackage[all]{xy}
\usepackage{amsthm,array,amssymb,amscd,amsfonts,latexsym, url}

\DeclareFontFamily{OT1}{wncyr}{\hyphenchar\font45}
\DeclareFontShape{OT1}{wncyr}{m}{n}{%
   <5> <6> <7> <8> <9> gen * wncyr
   <10> <10.95> <12> <14.4> <17.28> <20.74>  <24.88>wncyr10}{}
\DeclareFontShape{OT1}{wncyr}{m}{it}{%
   <5> <6> <7> <8> <9> gen * wncyi
   <10> <10.95> <12> <14.4> <17.28> <20.74> <24.88> wncyi10}{}
\DeclareFontShape{OT1}{wncyr}{m}{sc}{%
   <5> <6> <7> <8> <9> <10> <10.95> <12> <14.4>
   <17.28> <20.74> <24.88>wncysc10}{}
\DeclareFontShape{OT1}{wncyr}{b}{n}{%
   <5> <6> <7> <8> <9> gen * wncyb
   <10> <10.95> <12> <14.4> <17.28> <20.74> <24.88>wncyb10}{}
\input cyracc.def 
\def\rus{\usefont{OT1}{wncyr}{m}{n}\cyracc\fontsize{9}{11pt}\selectfont}

\DeclareMathSizes{9}{9}{7}{5} 

{\catcode`\@=11
\gdef\n@te#1#2{\leavevmode\vadjust{%
 {\setbox\z@\hbox to\z@{\strut#1}%
  \setbox\z@\hbox{\raise\dp\strutbox\box\z@}\ht\z@=\z@\dp\z@=\z@%
  #2\box\z@}}}
\gdef\leftnote#1{\n@te{\hss#1\quad}{}}
\gdef\rightnote#1{\n@te{\quad\kern-\leftskip#1\hss}{\moveright\hsize}}
\gdef\?{\bbFN@\qumark}
\gdef\qumark{\ifx\next"\DN@"##1"{\leftnote{\rm##1}}\else
 \DN@{\leftnote{\rm??}}\fi{\rm??}\next@}}

\def\oii{\overset{\simeq}\to}

\title[Patching and  local-global principles]{Patching and local-global principles for
homogeneous spaces  over function fields of $p$-adic curves}

\newtheorem{theorem}{Theorem}[section]
\newtheorem{thm}[theorem]{Theorem}
\newtheorem{prop}[theorem]{Proposition}
\newtheorem{lem}[theorem]{Lemma}
\newtheorem{cor}[theorem]{Corollary}

\theoremstyle{definition}

\newtheorem{remark}[theorem]{Remark}

\numberwithin{equation}{section}

\begin{document}

\newcommand{\Symp}{\mbox{\boldmath$\rm Sp$}}
\newcommand{\g}{\mathfrak{g}}
\newcommand{\el}{\mathfrak{l}}
\newcommand{\lt}{\mathfrak{t}}
\newcommand{\ls}{\mathfrak{s}}
\newcommand{\lc}{\mathfrak{c}}
\newcommand{\lu}{\mathfrak{u}}
\newcommand{\lr}{\mathfrak{r}}
\newcommand{\pr}{\operatorname{pr}}
\newcommand{\Hom}{\operatorname{Hom}}
\newcommand{\Rad}{\operatorname{Rad}}
\newcommand{\sign}{\operatorname{sign}}
 
\newcommand{\ve}{{\varepsilon}}
\newcommand{\vp}{{\varpi}}

\newcommand{\kbar}{\overline k}

\newcommand{\sdp}{\mathbin{{>}\!{\triangleleft}}} 
\newcommand{\bbAlt}{\operatorname{A}}   
\newcommand{\GL}{\operatorname{GL}}
\newcommand{\bbPGL}{\operatorname{PGL}}
\newcommand{\SL}{\operatorname{SL}}
\newcommand{\SU}{\operatorname{SU}}
\newcommand{\SO}{\operatorname{SO}}
\newcommand{\bbAd}{\operatorname{Ad}}
\newcommand{\ad}{\operatorname{ad}}

\newcommand{\rank}{\operatorname{rank}}
\newcommand{\bbAut}{\operatorname{Aut}}
\newcommand{\Char}{\operatorname{\rm char\,}} 
\newcommand{\Gal}{\operatorname{Gal}}
\newcommand{\galois}{\Gal}
\newcommand{\rto}{\dasharrow}
\newcommand{\M}{\operatorname{M}}        
\newcommand{\ord}{\mathop{\rm ord}\nolimits}
\newcommand{\Sym}{{\operatorname{S}}}    
\newcommand{\tr}{\operatorname{\rm tr}}
\newcommand{\trace}{\tr}

\newcommand{\Res}{\operatorname{Res}}
\newcommand{\Sha}{\mbox{\rus{\fontsize{11}{11pt}\selectfont{SH}}}}
\newcommand{\G}{\mathcal{G}}
\renewcommand{\H}{\mathcal{H}}
\newcommand{\gen}[1]{\langle{#1}\rangle}
\renewcommand{\O}{\mathcal{O}}
\newcommand{\C}{\mathcal{C}}
\newcommand{\Ind}{\operatorname{Ind}}
\newcommand{\End}{\operatorname{End}}
\newcommand{\T}{\mathbf G}
\newcommand{\GT}{\mbox{\boldmath$\rm T$}}
\newcommand{\Inf}{\operatorname{Inf}}
\newcommand{\Tor}{\operatorname{Tor}}
\newcommand{\m}{\mbox{\boldmath$\mu$}}

\newcommand{\Lbd}{{\sf \Lambda}}

\newcommand{\Id}{\operatorname{Id}}
\newcommand{\id}{\operatorname{id}}
\newcommand{\Nrd}{\operatorname{Nrd}}
\newcommand{\Trd}{\operatorname{Trd}}

\newcommand{\bbA}{{\mathbb A}}
\newcommand{\bbG}{{\mathbb G}}
\newcommand{\bbC}{{\mathbb C}}
\newcommand{\bbZ}{{\mathbb Z}}
\newcommand{\bbP}{{\mathbb P}}
\newcommand{\bbQ}{{\mathbb Q}}
\newcommand{\bbF}{{\mathbb F}}
\newcommand{\bbR}{{\mathbb R}}

\newcommand{\ssetminus}{\! \setminus \!}

\newcommand{\Ker}{\operatorname{Ker}}
\newcommand{\Spec}{\operatorname{Spec}}
\newcommand{\bbPGLn}{{\operatorname{PGL}_n}}
\newcommand{\bbPGLp}  {{\operatorname{PGL}_p}}
\newcommand{\Spin}{\operatorname{Spin}}

\newcommand{\Sympl}{{\operatorname{Sp}}}
\newcommand{\Stab}{\operatorname{Stab}}
\newcommand{\Span}{\operatorname{Span}}
\newcommand{\diag}{\operatorname{diag}}
\newcommand{\Galois}{\Gal}
\newcommand{\Mat}{{\operatorname{M}}}
\newcommand{\Mn}{\Mat_n}
\newcommand{\Ima}{\operatorname{Im}}
\newcommand{\Int}{\operatorname{Int}}
\newcommand{\trdeg}{\operatorname{trdeg}}
\newcommand{\Tr}{\operatorname{Tr}}
\newcommand{\N}{\operatorname{N}}
\newcommand{\sln}{\operatorname{sl}_n}
\newcommand{\cal}{\mathcal}
\newcommand{\Lie}{\operatorname{Lie}}
\newcommand{\bbAss}{\operatorname{Ass}}

\newcommand{\Br}{\operatorname{Br}}
\newcommand{\bbPic}{\operatorname{Pic}}
\newcommand{\Brnr}{\operatorname{Br}_{\text{\rm{nr}}}}
\newcommand{\rk}{\operatorname{rk}}
\newcommand{ \X}{{\cal X}}
 
\newcommand{\alp}{\alpha}
\newcommand{\eps}{\varepsilon}

\newcommand{\VG}{V{}_{{}^{\overset{}

\newcommand{\riso}{\hskip1mm {\buildrel \simeq \over \rightarrow} \hskip1mm}
\newcommand{\liso}{\hskip1mm {\buildrel \simeq \over \leftarrow} \hskip1mm}

{\centerdot}}}
  \hskip .5mm G\hskip
-3.9mm^{\overset{\centerdot}{}} \hskip
-1.37mm{}^{\underset{\centerdot}{}}\hskip 3.5mm{} }
 
\newcommand{\XH}{X{}_{{}^{\overset{}
{\centerdot}}}
  \hskip .5mm H\hskip
-3.9mm^{\overset{\centerdot}{}} \hskip
-1.58mm{}^{\underset{\centerdot}{}}\hskip 3.5mm{} }

\author{J.-L. Colliot-Th\'el\`ene}
\address{C.N.R.S.,
UMR 8628, Math\'ematiques, B\^atiment 425,
Universit\'e Paris-Sud, F-91405 Orsay, France}
\email{jlct@math.u-psud.fr}

\author{R. Parimala}
\address{Emory University, 400 Dowman Drive, Atlanta, Georgia 30322, USA}
\email{parimala@mathcs.emory.edu }

\author{V. Suresh}
\address{Department of Mathematics and Statistics, University of Hyderabad, Ghacibowli, Hyderabad 500046, Andhra Pradesh, India} 
\email{vssm@uohyd.ernet.in}


\begin{abstract}  
Let $F=K(X)$ be the function field of a smooth projective curve over a $p$-adic field $K$.
To each rank one discrete valuation of $F$ one may associate the completion $F_v$.
Given an $F$-variety $Y$ which is a homogeneous space of a connected reductive
 group $G$ over $F$, one may wonder whether the existence of $F_v$-points
on $Y$ for each $v$ is enough to ensure that $Y$ has an $F$-point.
In this paper  we prove such a result in two cases :

(i) $Y$ is a smooth projective quadric and $p$ is odd. 

(ii)     The group $G$ is the extension
of a reductive group over the ring of integers of $K$, and $Y$ is a principal homogeneous space of $G$.

An essential use is made of recent patching results of Harbater, Hartmann and Krashen.
There is a connection to injectivity properties of the Rost invariant and a result of 
Kat\^o.
\end{abstract}

\maketitle

\section{Introduction}
\label{intro}

Let   $K$ be a $p$-adic field, by which we mean a finite extension
of a field $\bbQ_{p}$. Let $A$ be its ring of integers.
Let $X/K$ be a smooth, projective,
geometrically integral curve. Let $F=K(X)$  be the function field of $X$.
This is a field of cohomological dimension 3.
Let $\Omega$ denote the set of discrete valuations (of rank one)
on the field $F$. Given $v \in \Omega$ we let $F_{v}$ denote the completion of
$F$ at $v$.

\medskip

We wonder whether in this context  there    is a local-global principle for
the existence of rational points on homogeneous spaces of connected
linear algebraic groups over $F$.

\medskip

{\it Conjecture 1}  Let $F=K(X)$ be as above.  
Let $Y/F$ be a projective homogeneous space of a connected linear
algebraic group. If $Y$  has points in all completions $F_{v}$,
then it has an $F$-rational point.
 
\medskip
  
{\it Conjecture  2}  Let $F=K(X)$ be as above. Let $G/F$ be a semisimple,
simply connected group.
If a class $\xi$ in the Galois cohomology set $H^1(F,G)$ has trivial image in
each $H^1(F_{v},G)$, then $\xi$ is trivial. In other words, if a principal
homogeneous space  under $G$ has points in all completions $F_{v}$,
then it has an $F$-rational point.

\medskip

It is unlikely that Conjecture 2 holds for an arbitrary connected
reductive group $G$, for instance for a torus. It definitely fails  for
$G$ a finite constant group, see \S 6.

\medskip

As we explain in Section \ref{RostKato} (Theorem \ref{rostkato}),  Conjecture 2
may be proved for most quasi\-split simply connected groups  by using a combination of
properties of the Rost invariant and a result of Kat\^o  \cite{Kato}.

\bigskip

In their recent paper  \cite{HHK}, Harbater, Hartmann and Krashen have developed
the patching technique of \cite{HH} to the point where they get local-global theorems
for homogeneous spaces. Their main local-global  theorems refer to some 
other families of overfields of $F$ than the family $\{F_{v}\}$ we consider here.
But they manage to apply the technique to the extent that
they give a radically new proof of the theorem by Parimala and Suresh \cite{PS2}
that any quadratic form in at least 9 variables over $K(X)$ ($K$ as above, nondyadic,
$X$ a curve over $K$) has a nontrivial zero.

Their techniques apply more generally to complete discrete valuation rings
with arbitrary residue field.

\medskip

In the present paper, we use the method and 
theorems of Harbater, Hartmann and Krashen
to prove the  following results.

1) For smooth quadrics of dimension at least 1, which are projective homogeneous spaces under the  associated special orthogonal group, under the assumption that the characteristic of the residue field of
$K$ is not 2, we prove Conjecture 1.  
We actually prove the more general result (Theorem  \ref{localglobalquadratic}) :

{\it Let $A$ be a complete
 discrete valuation ring with fraction field $K$
and residue field $k$ of characteristic different from 2.
Let $X$ be a smooth, projective, geometrically integral
curve over $K$. Let $ F=K(X)$ be the function field of $X$. 
Let $q$ be a nondegenerate 
 quadratic form  over $F$ in at least 3 variables. If for each discrete valuation  
 of $F$, the form $q$ is isotropic over 
 the completion of $F$ with respect to this valuation, then $q$ is isotropic
 over $F$.}

2)  We show (Theorem \ref{mainthm}) that the statement of
Conjecture 2 holds for any (fibrewise connected) reductive $A$-group $G$.

This relies on the following general result (Theorem \ref{generalizedGroth}) :

{\it Let $A$ be a complete discrete valuation ring, $K$ its field of fractions
and $k$ its  
residue field.  
Let $\X/A$ be a projective, flat curve over ${\rm Spec} \ A$. 
Assume that $\X$
is connected and regular.
Let $F$ be the function field of $X$.
Let $\Omega$ be the set of all discrete valuations on $F$.
Let $G/A$ be a  (fibrewise connected) reductive group.
If  
 there exists a connected linear algebraic group $H/F$
such that the $F$-group $ (G\times_{A}F)\times_{F} H $ is an $F$-rational variety,
 then
the restriction map with respect to completions
$ H^1(F,G) \to \prod_{v \in \Omega} H^1(F_{v},G)$
has a trivial kernel.}

As mentioned above, an independent argument, which  builds upon injectivity properties of the Rost invariant (which themselves rely on a case by case proof)
and upon a theorem of Kat\^o, yields a proof of Conjecture 2 for quasisplit, absolutely simple,
simply connected   groups over $F$  with no $E_{8}$-factor.

In the final Theorem \ref{rostinjectifK(X)}, we revert the process: we use Theorem \ref{mainthm} together with Bruhat-Tits theory
to discuss  the triviality of the kernel of the Rost invariant for split simply connected groups over
a function field in one variable over a $p$-adic field. The result is classificationfree; in particular, it applies
to $E_{8}$.

\medskip

Throughout this paper, when we write ``discrete valuation ring'', we mean ``discrete valuation ring of rank one'',
and when we write discrete valuation we mean valuation with value group $\bbZ$.

\bigskip

\section{Why the $u$-invariant should behave well for function fields over the $p$-adics}

The $u$-invariant of a field is the maximal dimension of anisotropic quadratic forms
over that field.
Let us start with some reminders from the paper \cite{KK} by Kat\^o and Kuzumaki.
 
 Let $r \geq 1$ be an integer. We say that a field $F$ is a $C_{r}^0$ field
 if the following condition holds : 
 
 {\it For any finite field extension $F'$ of $F$ and any integers $d \geq 1$ and $n > d^r$,
 for any homogeneous form over $F'$ of degree $d$ in $n$ variables, the g.c.d.
 of the degrees of finite field extensions $F''/F'$ over which the form acquires a nontrivial
 zero is 1.}

 The condition amounts to requiring that the $F'$-hypersurface defined by the form
 contain a zero-cycle of degree 1 over $F'$.

 Assume  ${\rm{char}}(F)=0$. For each prime $l$, let $F_{l}$ be the 
 fixed field of a pro-$l$-Sylow subgroup of the absolute Galois group of $F$.
 Any finite subextension of $F_{l}/F$ is of degree prime to $l$.

 The field $F$ is $C_{r}^0$ if and only if each of the fields $F_{l}$ is $C_{r}$
 in the usual sense (\cite[Lemma 1]{KK}).
 A finite field extension of a $C_{r}^0$-field is $C_{r}^0$.
 The following easy lemma does not appear in \cite{KK}.

 \begin{lem}
 Let $F$ be a field of characteristic zero. If $F$ is $C_{r}^0$
 then a function field $E=F(X)$ in $s$ variables over $F$ is $C_{r+s}^0$.
 \end{lem}
 
 \begin{proof} Let $E'$ be a finite field extension of $E$.
 After replacing $F$ by a finite extension, which by assumption is still $C_{r}^0$,
 we may assume that $E'$ is the function field  $F(X)$ of
 a geometrically integral $F$-variety $X$.  The field $F_{l}$ is $C_{r}$, hence by the classical
 transitivity properties (Lang, Nagata), the  field $F_{l}(X)$, function field of $X\times_{F}F_{l}$,
 is a $C_{r+s}$-field. Thus any form of degree $d$ over $F(X)$ in $n > d^{r+s}$ variables
 has nontrivial solutions in $F_{l}(X)$, hence in a finite extension of $F(X)$ of degree prime to $l$.
 As this applies to each prime  $l$, this concludes the proof.
 \end{proof}

It is an open question whether  $p$-adic fields have the
 $C^0_{2}$-property.
An equicharacteristic  zero  analogue of that statement 
is proven in \cite{CT}.

\begin{prop}
Assume that 
 $p$-adic fields have the $C^0_{2}$-property.
Then over  any function field $K(X)$ of transcendance degree $r$
over a $p$-adic field $K$, any quadratic form in strictly more than $2^{2+r}$
variables has a nontrivial zero.
\end{prop}

\begin{proof}
By the previous lemma, such a quadratic form has a nontrivial zero
in an extension of odd degree of the field $K(X)$. By a theorem of Springer
  \cite[VII, Thm.  2.3]{Lam}
 this implies that the quadratic form has a nontrivial
zero in $K(X)$.
\end{proof}

\section{A local-global principle for isotropy of quadratic forms }

\begin{thm}\label{localglobalquadratic}
Let $A$ be a complete
 discrete valuation ring with fraction field $K$
and residue field $k$ of characteristic different from 2.
Let $X$ be a smooth, projective, geometrically integral
curve over $K$. Let $ F=K(X)$ be the function field of $X$. 
Let $q$ be a nondegenerate 
 quadratic form  over $F$ in at least 3 variables. If for each discrete valuation  
 of $F$, the form $q$ is isotropic over 
 the completion of $F$ with respect to this valuation, then $q$ is isotropic
 over $F$.
\end{thm}

\begin{proof}
Suppose we are given a diagonal quadratic form $q=<a_1,\cdots,a_n>$ over $F=K(X)$
which is isotropic over the field of fractions of the completion of any
discrete valuation ring of $F$.

 Let us recall basic notation from \cite{HH} and \cite{HHK}.

Let $t$ denote a uniformizing parameter for $A$.

One may choose a regular proper model $\X/A$ of $X/K$ such that
there exists a reduced divisor $D$ with strict normal crossings
which contains both the support of the divisor of the $a_{i}$'s
and the components of the special fibre of $X/A$.
Let $Y=\X\times_{A}k$ denote the special fibre.

For the  generic point $x_{i}$ of an irreducible  component  $Y_{i}$ of   $Y$, there is an affine Zariski neighbourhood $W_{i}$ of  
$x_{i}$
 in $\X$ such that 
the restriction of $Y_{i}$ to $W_{i}$ is a principal divisor.  

Let $S_{0}$ be a finite set of closed points of the special fibre
containing all singular points of $D$ and  all points which lie on some $Y_{i}$ but not in 
the corresponding 
$W_{i}$.

Choose a finite $A$-morphism $f :\X \to \bbP^1_{A}$ as in \cite[Prop. 6.6]{HH}.
In particular, we have the following three properties.
 The set $S_{0}$ is contained in   $S$,  the inverse image under $f$  of the point at infinity
of the special fibre $\bbP^1_{k}$.  All the intersection points of two components $Y_i$ are in $S$.
 Each component $Y_{i}$ contains at least one point of $S$.

   Let $U \subset Y$ run through the
reduced, irreducible components of the complement of $S$ in $Y$.
Each  $U$ is a regular affine irreducible curve  over $k$. 
Let $k[U]$ be the ring of regular functions on this curve.
This is a Dedekind domain.
We thus have $U= \Spec k[U]$. Let $k(U)$ denote
the fraction field of $k[U]$.

Each $U$ is included in an 
open affine subscheme  $\Spec R^U$   of $\X$
and is a principal effective divisor in this affine subscheme.

By definition, the ring $R_U$ is  the ring of elements in $F$
which are regular on $U$. It is a regular ring,
as a direct limit of regular rings. The ring $R_{U}$
is a localisation of $R^U$. Thus $U$ is a principal effective divisor
on $\Spec R_{U}$. It is given by the vanishing of an element $s \in R_{U}$.

The $t$-adic completion
${\hat R}_U$ of $R_U$ is a domain (\cite{HHK}, Notation 3.3).
We have $t=u.s^r$ for some integer $r \geq 1$ and a unit $u \in R_{U}^{\times}$.
Thus  the $t$-adic completion ${\hat R}_U$   coincides with the $s$-adic completion
of $R_{U}$. 
By definition, $F_U$ is the field of fractions of ${\hat R}_U$. We have $k[U]=R_{U}/s= {\hat R}_U/s$. 

For $P\in S$, the completion  $\hat{R}_P$
   of the local ring $R_P$ of $\X$ at $P$ is a domain
    (\cite{HHK}, Notation 3.3). By definition, the field
$F_P$ is  
 the field of fractions of $\hat{R}_P$.

For $p=(U,P)$  a pair with   $P \in S$ in the closure
  of   an irreducible component $U$ of the complement of $S$ in $Y$,
  one lets $R_{p}$ be the
 complete discrete valuation ring which is  the completion of the
 localisation of  $\hat{R}_P$ at the height one prime ideal
 corresponding to $U$. By definition, the field $F_p$ is the field of fractions of $R_p$.

By  \cite[Prop. 6.3]{HH}, the field $F$ is the inverse limit of the finite inverse
system of fields $\{F_U,F_P,F_p\}$.

\medskip

Let us show that $q$ is isotropic over each field $F_U$.
 
Each diagonal
entry $a_i$ of the form $q$    is supported only along $U$ in $\Spec R_U$, thus
is of the form $u.s^j$  where $u$ is a unit in $R_U$. Hence   the
 quadratic form $q$ over $F$ is isomorphic to a quadratic form over $F$ of the
shape  $$<b_1,\cdots,b_{\rho},s.c_1, \cdots,  s.c_{\sigma}>$$ where $b_i$
and $c_j$ are units in $R_U$.   

By hypothesis, $q$ is isotropic over the field of fractions of the
completed local ring  of $\X$ at the generic point of $U$.
By a theorem of  Springer \cite[VI, Prop. 1.9 (2)]{Lam}, this implies that
the image of
at least one of the two forms
$ {q}_1=<b_1,\cdots, b_{\rho}>$ or  $ {q}_2=<c_1,\cdots,c_{\sigma}>$
under the composite  homomorphism $R_{U} \to k[U] \hookrightarrow  k(U)$ is isotropic over $k(U)$.

Since the residue characteristic is not 2, each of the
forms $q_{1}$ and  $q_{2}$ defines a smooth quadric over $R_{U}$.
In particular each of them defines a smooth quadric over $k[U]$.
Since $k[U]$ is a Dedekind domain, if such a projective quadric has
a point over $k(U)$, it has a point over $k[U]$. 
Since the quadric is smooth over $R_{U}$,
a $k[U]$-point lifts to an ${\hat R}_U$-point
(compare the discussion after \cite[Lemma 4.5]{HHK}).
Thus $q$ has a nontrivial zero over $F_{U}$.

Let us show that $q$ is isotropic over each field $F_P$.
Let $P \in S$. The local ring $R_P$ of $\X$ at $P$ is regular.
Its maximal ideal is generated by two elements $(x,y)$
with the property that any $a_i$ is the product of a unit, a power of $x$
and a power of $y$. Thus over the fraction field 
$F$ 
of $R_P$,
the form $q$ is isomorphic to a form $$q_1 \perp x.q_2 \perp y.q_3 \perp xy.q_4$$
where each $q_i$ is a nondegenerate diagonal quadratic form over $R_P$.
Let $R_y$ be the localization of   $R_P$ at the prime ideal $(y)$. 
This is a
discrete valuation ring with fraction field $F$ 
and with residue field
$E$ the field of fractions of the discrete valuation ring $R_P/(y)$.
By hypothesis,  the form $(q_1\perp x.q_2) \perp y.(q_3 \perp xq_4)$ is isotropic over
the field of fractions of the completion of $R_y$. By Springer's theorem
\cite[VI, Prop.  1.9 (2)]{Lam},
this implies that over $E$ the reduction of
either  $(q_1\perp x.q_2)$ or $(q_3 \perp xq_4)$ is isotropic.
Since $x$ is a uniformizing parameter for $R_P/(y)$,
by Springer's theorem \cite[VI, Prop.  1.9 (2)]{Lam},
this last statement implies that over the residue field
of $R_P/(y)$, the reduction of one of the forms $q_1,q_2,q_3,q_4$
is isotropic. But then one of these forms is isotropic over $\hat{R}_P$,
hence over the field $F_P$ which is the fraction field of $\hat{R}_P$.

 \medskip

 The quadric $Z/F$ defined by the vanishing of $q$ is a homogeneous space
 of the group $\SO(q)$ over $F$, which since $q$ 
 is of rank at least 3 is 
 a connected group.
 By Witt's result, for any field $L$ containing $F$,
 the group $\SO(q)(L)$ acts transitively on $Z(L)$.
 The $F$-variety underlying $\SO(q)$ is   $F$-rational
 (Cayley parametrization).  We have $Z(F_{U}) \neq \emptyset$
 for each $U$ and $Z(F_{P})\neq \emptyset$ for each $P \in S$.
By   \cite[Thm.~Ê3.7]{HHK}, we get $Z(F) \neq \emptyset$.
 \end{proof}

\begin{remark}\label{remarkmodel}
Note that in the proof the only discrete valuation rings 
 which are
used are the local rings at a point of codimension 1 on a {\it suitable} regular proper
model $\X$ of $X$. 
See however Remark \ref{notsuitable}.
\end{remark}

\begin{remark}\label{notwithtwo}
The theorem does not extend to forms in $2$ variables. See Remark \ref{totallysplit}  
and 
Section \ref{Appendix}. 
\end{remark}

 The following corollary is a variant of  a  theorem of Harbater, Hartmann and Krashen \cite[Thm. 4.10]{HHK}.

\begin{cor}\label{induction}
Let $A$ be a complete discrete valuation ring with fraction field $K$
and residue field $k$ of characteristic different from 2. Let $r \geq 1$ be an integer.
Assume that 
any quadratic form in strictly more than $2r$ variables over any function field in one 
variable over $k$
is isotropic.

Then any quadratic form in strictly more than $4r$ 
variables over the function field $F=K(X)$ of a curve $X/K$  is isotropic.
\end{cor}
\begin{proof}
Let $L$ be a finite field extension of $K$.
This is a complete discrete value field with residue field
a finite extension $l$ of $k$. The hypothesis made on quadratic
forms over functions fields in one variable over $k$, in particular 
quadratic forms over the field $l(t)$,
and Springer's theorem 
(\cite[VI, Cor. 1.10]{Lam})
applied to the field  $l((t))$  
imply that any quadratic form  in strictly more than $r$ variables over $l$
has a zero. Another application of Springer's theorem then
implies that any quadratic form in strictly more than $2r$ variables over $L$
is isotropic.

 Let $q$ be a quadratic form in $n$ variables over $F$ with 
$ n > 4r$. By Theorem \ref {localglobalquadratic} and Remark
\ref{remarkmodel}, to prove the corollary
  it suffices to show that $q$ is 
isotropic over $F_v$ for every discrete valuation $v$ with 
residue field either a function field in one variable over $k$
or a finite extension of $K$. By the hypothesis, the preceding paragraph and Springer's
theorem, the quadratic form $q$ is isotropic over such an $F_v$.
\end{proof}

Corollary  \ref{induction} in its turn is a generalization of the main result of \cite{PS2}:

\begin{cor}
If $K$ is a nondyadic $p$-adic field,
any quadratic form in at least 9 variables over a  function field in one variable $K(X)$
has a nontrivial zero.
\end{cor}

\begin{remark}\label{notsuitable}
{\it In Theorem \ref{localglobalquadratic}, it is not enough to consider the discrete valuation rings corresponding
to the codimension 1 points of a given regular proper  model $\X/A$.}

  Let $p$ be an odd prime and $a$ a unit in $\bbZ_p$ which is not a square.
Rowen, Sivatski, Tignol  \cite[Cor. 5.3]{RST} (see also \cite{JT})
 have shown that the tensor
product  $ D=(a,p) \otimes (t, a(p-t))$ of
quaternion algebras over $F=\bbQ_p(t)$ is a division algebra. 
The tensor product $(a,b) \otimes (c,d)$ of two quaternion algebras over a field $F$
(${\rm char} \ F \neq 2$)
is a division algebra, i.e. has index 4,
 if and only if the associated Albert form  $<-a,-b,ab, c,d,-cd>$ is 
anisotropic over $F$. Thus 
the  quadratic form $q=<-a, -p, ap, t, a(p-t), -at(p-t)>$ is   a 6-dimensional
anisotropic quadratic form over $F=\bbQ_p(t)$.

Let $\X=\bbP^1_{\bbZ_{p}}$ be the projective line over $\bbZ_p$. The codimension one
points $v$ of $\X$ are  given by   irreducible monic polynomials in $\bbQ_p[t]$,
by $1/t$ and by the height one prime ideal 
of
$\bbZ_p[t]$ generated by $p$. 
Let $F=\bbQ_{p}(t)$, and let $F_{v}$ denote the completion of $F$
at a discrete valuation $v$ of $F$.

The residue field at point $v$ of codimension 1 of $\X$ is either a $p$-adic field
or $\bbF_{p}(t)$. Any quadratic form in at least 5 variables over such a field
is isotropic. At any prime $v$ of codimension 1 of $\X$ different from $p,t,(p-t), 1/t$,
 the form $q$ therefore has a zero over $F_{v}$.
 At $v=t$, one of the residue forms is 
 $<-a,-p,ap,ap>$
which is isotropic over the residue field $\bbQ_{p}$, since $<-1,a,a>$ is.
At $v=t-p$, one of the residue forms of $q$ is $<-a,-p,ap,p>$
and this form is clearly isotropic.
At $v=(1/t)$, one of the residue forms of $q$ is $<-a,-p,ap,a>$
 which is clearly isotropic.
At $v=p$, one of the residue forms over the field $\bbF_{p}(t)$
 is $<-a,t,-at,a>$ which again is clearly isotropic.
Thus the quadratic form $q$ is isotropic over each field $F_{v}$
corresponding to a point of codimension 1 on $\X$.

Theorem \ref{localglobalquadratic} and   the result of  \cite{RST} show that there
must exist another completion $F_{v}$, corresponding to a codimension 1
point on another model of $\bbP^1_{\bbZ_{p}}$, at which the form is
anisotropic.

Note that on $\bbP^1_{\bbZ_{p}}$, the divisor 
associated to the quadratic form does not have normal crossings
at the point defined by the ideal $(p,t) \subset \bbZ_{p}[t]$
(compare the proof of Theorem \ref{localglobalquadratic}).
It is thus natural to blow up the corresponding point. In practice,
one introduces a new variable $x$ and one sets $t=px$.
The quadratic form $q$ now reads $<-a,-p,ap,px,ap(1-x),-ax(1-x)>$.
At the prime ideal $p$ of $\bbZ_{p}[x]$, with residue field $\bbF_{p}(x)$,
the two residue forms are $<-a,-a(1-x)>$ and  $<-1,a,x,a(1-x)>$.
Since $x-1$ is not a square in $\bbF_{p}(x)$, the first form is clearly anisotropic.
 As for the second one, 
the two residue forms of  $<-1,a,x,a(1-x)>$  at  the valuation of $\bbF_{p}(x)$  with uniformizing 
parameter $1/x$ are anisotropic over $\bbF_{p}$, because 
  $a$ is not a square. 
  
  At any closed point of $\bbP^1_{\bbZ_{p}}$ different from the point defined by
 $(p,t) \subset \bbZ_{p}[t]$ the form $q$ admits a reduction of the shape $<a,-a>$,
 hence it is isotropic over the fraction field of the complete  local ring at such a point.
 The $\bbZ_{p} $-homomorphism  $\bbZ_{p}[t] \to   \bbZ_{p}[x]$ sending $t$ to $px$
 sends the ideal $(p,t)$ of $\bbZ_{p}[t] $  to the ideal $p$ of  $\bbZ_{p}[x]$.
 It    induces an injective   homomorphism of the corresponding complete  local rings,
 hence an embedding of their fraction fields.
 The above argument shows that $q$ is anisotropic over the bigger fraction field.
It is is thus anisotropic over the fraction field of the completion of  $\bbZ_{p}[t] $
at the maximal ideal $(p,t)$.
  \end{remark}

\begin{remark}

 {\it The following question was raised by D. Harbater. Let $A,K,X$ and $F=K(X)$ be
 as in Theorem \ref{localglobalquadratic}.
Suppose a projective homogeneous variety $Z$ over $F$
under an $F$-rational connected linear algebraic group
has points  in the field of fractions of the completions at closed points of all possible  regular proper models of $X$ over $A$. Does $Z$ admit a rational point over $F$?
 The following example gives a  negative answer to this question, already with $Z$ a quadric.}

Let $p$ be an odd prime, let ${\mathcal X}/\bbZ_{p}$ be a smooth curve over
$\bbZ_{p}$ of relative genus at least 1. 
Let $X/\bbQ_{p}$ be its generic fibre and let
 $Y/\bbF_{p}$ be its special fibre. There exist two quaternion division algebras
$H_{1}$ and $H_{2}$ over the function field $\bbF_{p}(Y)$ 
whose ramification loci on $Y$ 
 are disjoint.
Let  
$\overline{q}_1$ and $\overline{q}_2$ be the 
reduced norms  attached to these two quaternion division algebras.
Let $q_1, q_2$  be lifts of these quadratic forms to the local ring $R$  of the
generic point of $Y$ on $\mathcal{X}$. 
 Consider the quadratic form $q=q_{1}\perp p.q_{2}$
over the function field $F=\bbQ_{p}(X)$. This form is isotropic over the fraction field
of the local ring of any closed point $P \in \mathcal{X}$. Indeed at any such point
either the form $q_1$ or the form $q_2$ has good reduction.

 On the other hand, the form $q$ is anisotropic over $F$ since
both $q_1$ and $q_2$ are anisotropic over the $p$-adic completion
of the local ring $R$, whose uniformizing parameter is $p$.

Since the genus of $X$ is  at least 1, by a well known result of
Shafarevich \cite{Sha} and Lichtenbaum \cite{Licht1},
 the curve 
${\mathcal X}/\bbZ_{p}$ is an absolute minimal model of $X$ over $\bbZ_{p}$.
Thus if $\mathcal{X}'/\bbZ_{p}$ is another model, there is a birational
$\bbZ_{p}$-morphism $\mathcal{X}' \to \mathcal{X}$.
The argument given above
then shows that  $q$ is isotropic on the fraction field of the local ring
of any closed point 
 of $\mathcal{X}' $.
\end{remark}

\begin{remark}

Let $F$ be a function field in one variable over a $p$-adic field $k$.

For quadratic forms in 3 or 4 variables, there is a refined local-global principle for isotropy of quadratic forms : 
one only needs to take into account
 discrete valuations which are  trivial on $K$.  The case of 3 variables is a consequence of
a theorem of Lichtenbaum \cite{Licht2}, based on Tate's duality theorem for abelian varieties over a
 $p$-adic field : {\it for $X/K$ a smooth projective geometrically connected curve $X$ over a $p$-adic
 field $K$, if an element of the Brauer group of $X$   vanishes after evaluation 
 at each closed point of $X$, then it is zero}. The case of forms in 4 variables follows from the case
 of 3 variables by passing over to the discriminant extension of the 4-dimensional form.
This  should be compared with Theorem  \ref{localglobalquadratic}. 
  As Remark \ref{notsuitable} shows, such a refined local-global principle does not hold for forms in 6 variables.  
It actually does not hold for forms in 5, 6, 7 or 8 variables, as the following argument shows.

Let $K$ be a $p$-adic field.
Suppose we are given a  smooth complete  intersection  $Y$ of two quadrics 
given by a system of two quadratic forms $f=g=0$ in projective space $\bbP^n_K$
such that  $Y(K)=\emptyset$.
By a theorem due independently to Amer (unpublished) and to Brumer \cite{Brumer},
the quadratic form $f+tg$ in $n+1$ variables over  the field $K(t)$  then does not have a nontrivial zero. 
 The hypothesis of smoothness of $Y$ ensures that  over any  completion $F_v$ of $F=K(t)$
at a place trivial on $K$, 
the form $f+tg$ contains a good reduction subform of rank at least  $n$.
Since $K$ is $p$-adic, for $n \geq 5$, such a form has a nontrivial zero.
Hence for $n \geq 5$, that is from 6 variables onwards, the form $f+tg$
over $F=K(t)$ has a nontrivial zero in each completion of $F$ at a place
trivial on $K$.  It remains to exhibit such systems of forms as above.
By a classical compacity argument, to prove the existence of such a smooth $Y$,
it is enough to produce an arbitrary complete intersection of two quadrics
in $\bbP^n_K$ with no $K$-point. But that is easy. Let $f(x_1,x_2,x_3,x_4)$ be the  norm form of 
the nontrivial quaternion algebra over $K$. Then the system
$f(x_1,x_2,x_3,x_4)=0, f(x_5,x_6,x_7,x_8)=0$ defines such a complete intersection
in $\bbP^7$, and one gets suitable systems in lower dimensional projective space
by letting some variables vanish.   With $f$ as above, $m$ a suitable integer, and $h_1$ and $h_2$ suitable diagonal quadratic forms, one can produce a smooth complete intersection without $K$-point of the shape
$$f(x_1,x_2,x_3,x_4)+p^mh_1(x_5,x_6,x_7,x_8)=0, \hskip3mm p^mh_2(x_1,x_2,x_3,x_4)+ f(x_5,x_6,x_7,x_8)=0.$$

 There also exist smooth intersections of two quadrics $$f(x_1,x_2,x_3,x_4,x_5)=g(x_1,x_2,x_3,x_4,x_5)=0$$  in $\bbP^4_{K}$ 
 such that at any completion  $F_v$ of $F=K(t)$ at a place trivial on $K$ 
 the form $f+tg$ has an $F_{v}$-point but has no $F$-point.
 Here is one example. Let $p$ be a prime, $p \equiv 1 \ {\rm mod} \ 4$.
 Let $K=\mathbb Q_{p}$, $u \in \mathbb Z_{p}^*$ a unit which is not a square and $s\geq 2$ an integer.
 Then take $$f=x_{1}^2+ux_{2}^2+px_{3}^2+up^{2s}x_{4}^2+p^{2s-2}x_{5}^2$$
 and $$g=p^{4s+1}x_{1}^2+p^{4s}x_{2}^2+up^{2s}x_{3}^2+ x_{4}^2+px_{5}^2.$$
 One immediately checks that this defines a smooth complete intersection  $Y$  in $\bbP^4_{\mathbb Q_{p}}$.
 The system $f=g=0$ has no primitive solution modulo $p^2$, hence $Y$ has no rational point
over $\mathbb Q_{p}$. Let $v$ be a place of $F$ trivial on $K$.
If $f+tg$ has good reduction at $v$, then its reduction has rank 5 over the residue field,
hence is isotropic. The places $v$ at which  $f+tg$ has bad reduction are 5 distinct rational points
of $\Spec \mathbb Q_{p}[t]$. At each of these completions, $f+tg$  has a good reduction subform whose
reduction is   isotropic   of rank 4 over $\mathbb Q_{p}$.

\end{remark}

\section{A local-global principle for principal homogeneous spaces under certain linear algebraic groups} \label{}

Given a scheme $X$ and a smooth $X$-group scheme $G$, we let $H^1(X,G)$ 
denote the first \v{C}ech cohomology set for the \'etale topology on $X$.

The following lemma is  known (\cite[Lemma 4.1.3]{H0}).
 
\begin{lem}\label{cartesiandensity}
Let $A$ be a discrete valuation ring, $K$ its fraction field, $\hat{A}$ its completion and
$\hat{K}$ the field of fractions of $\hat{A}$. Let $G/A$ be a reductive group (with  connected fibres).  
 Then the fibre product of $H^1(K,G)$ and $H^1(\hat{A},G)$ over $H^1(\hat{K},G)$
is $H^1(A,G)$. 
\end{lem}
\begin{proof}
Let $G \subset {\rm GL}_{n,A}$ be a closed embedding  of $A$-groups and let $Z/A$ denote the
quotient ${\rm GL}_{n,A}/G$ (see  \cite[Cor. 6.12]{CTSa}).
For any local $A$-algebra $B$, by \cite[III, 3.2.4 and 3.2.5]{Giraud} we have an exact sequence of pointed sets
$$ {\rm GL}_{n}(B) \to Z(B) \to H^1(B,G) \to 1.$$
More precisely, the natural map $Z(B) \to H^1(B,G)$ induces a bijection
between the quotient of $Z(B)$  by the left action of $GL_{n}(B)$
and the set $H^1(B,G)$.
That the natural map $H^1(A,G) \to H^1(K,G)$ is injective is a theorem of Nisnevich
\cite{Nis}: there it is proven that the kernel is trivial for any reductive $A$-group $G$,
a known twisting argument (an \'etale cohomology variant of \cite{SCG} \S 5.4 p.~47, or \cite{KMRT} 28.9 p.~367)
 then implies that the map is injective.
Let  $x \in Z(K)$ be a lift of $\xi \in H^1(K,G)$. If the image of $\xi$
in $H^1(\hat{K},G)$ comes from $H^1(\hat{A},G)$, then there exists
 $y \in Z(\hat{A})$ and $\rho \in {\rm GL}_{n}(\hat{K})$ such that $ \rho.x= y \in Z(\hat{K})$.
 The set ${\rm GL}_{n}(K)$ is dense in ${\rm GL}_{n}(\hat{K})$, the map 
 ${\rm GL}_{n}(\hat{K}) \to Z(\hat{K}) $ is continuous, and $Z(\hat{A})$ is open in $Z(\hat{K})$.
 We may thus find $g \in {\rm GL}_{n}(K)$ close enough to $\rho$ that $g.x$ lies in 
 $Z(\hat{A})$. Now $g.x$ lies in $Z(K) \cap Z(\hat{A}) = Z(A)$.
\end{proof}
 
\bigskip
 
{\it  Unramified classes}
 
 \smallskip
 Given $G/A$ as above, and $\xi_{K} \in H^1(K,G)$, one says that $\xi_{K}$
 is unramified at $A$ if it lies in the image of $H^1(A,G)$. 
 By the above lemma, it then comes from a well defined element $\xi_{A} \in H^1(A,G)$.
 By the same lemma, the condition is equivalent to requiring that the image of
 $\xi_{K}$ in $  H^1(\hat{K},G)$ comes from a well defined element  $\xi_{\hat{A}} \in    H^1(\hat{A},G)$.
 Let $k$ denote the residue class field of $A$ and $\hat{A}$.
 If a class $\xi_{K} \in H^1(K,G)$ is unramified, one may then consider
 its evaluation $\xi_{k} \in H^1(k,G_{k})$. It is given by the image of
 $\xi_{A}$, or  of $\xi_{\hat{A}}$, in $H^1(k,G_{k})$.

\begin{thm}\label{mainprop} Let $A$ be a complete discrete valuation ring, $K$ its field of fractions
and $k$ its  
residue field.  
Let $\X/A$ be a projective, flat curve over ${\rm Spec} \ A$. Assume that $\X$
is connected and regular. Let $F$ be the function field of $\X$. Let $G/A$ be a   (fibrewise connected)
reductive group. Let $\xi_F \in H^1(F,G)$ be a class which is unramified at 
all
codimension 1
points of $\X$.

(i)  There exists $\xi \in H^1(\X,G)$ whose image in
$H^1(F,G)$ is $\xi_{F}$.

(ii) If moreover  the (reduced) components of the special fibre
are regular, and for any such component $Y$  the image   $\xi_{k(Y)}$ in $H^1(k(Y),G)$ is trivial,
then  at any point $P$ of codimension 1 or 2 of $\X$, with residue field $\kappa(P)$,
 the image
$\xi_{P} \in H^1(\kappa(P),G)$ is trivial.

(iii) If moreover
 there exists a connected linear algebraic group $H/F$
such that the $F$-group $ (G\times_{A}F)\times_{F} H $ is an $F$-rational variety,
then
 $\xi_{F} =1 \in H^1(F,G)$.\end{thm}

\begin{proof}
(i)  By definition of an unramified class, for each point $P$ of codimension 1
on $\X$ there exists $\xi_P \in H^1(O_{\X,P},G)$ with image $\xi_{F}$
over the fraction field $F$ of $O_{\X,P}$. By   Nisnevich's theorem \cite{Nis},
the class $\xi_P $ is uniquely defined.

There then exists a Zariski open set  $V \subset \X$
which contains all points of dimension 1 of $\X$ and an element of $H^1(V,G)$
with image $\xi$ in $H^1(F,G)$ (see the proof  of  \cite[Prop. 6.8]{CTSa}).
 Since $\X$ is regular and of dimension 2,
  \cite[Thm. 6.13]{CTSa}
 shows that one may take $V=\X$.
We thus have a class $\xi \in H^1(\X,G)$ with image $\xi_F$ in $H^1(F,G)$.
In other terms, we have a torsor $E$ over $\X$ under the $\X$-group scheme $G_{\X}=G\times_{A}\X$,
whose restriction over the generic point of $\X$  has  class $\xi_{F} \in H^1(F,G)$.

\medskip

(ii) Let   $P$ be  a closed point of the special fibre.
Let $k(P)$ denote the residue field at $P$.
Let $Y$ be a  component of the special fibre which contains $P$.
Since we assumed the components to be regular,
the local ring $O_{Y,P}$ is a discrete valuation ring.
The image of the class of $\xi $ in $H^1(O_{Y,P},G)$
is now trivial because  its image in $H^1(k(Y),G)$ is trivial (easy case of \cite{Nis}).
Hence the image of $\xi$ in $H^1(k(P),G)$ is trivial.

Now let $P$ be a closed point of the generic fibre. 
Let $B$ denote the integral closure of $A \subset K$ in the residue field $K(P)$.
This is a complete discrete valuation ring.
There exists
an $A$-morphism   $\Spec  B \to \X$ which extends the inclusion
of $P$ in $X$. The image of the special point of $\Spec B$
is a point in the special fibre of $\X/A$, hence it lies on some component $Y$. By the above
argument, the evaluation of $\xi$  at that point is trivial.
Thus the inverse image of  $\xi$ on $\Spec B$ is a $G$-torsor with trivial
special fibre. By Hensel's lemma on the complete local ring $B$,
this inverse image is trivial. Hence $\xi_{P}$  is trivial at  $P$, which is the generic point
of $\Spec B$.

\medskip

(iii) The hypothesis in (ii)  ensures that for each component $Y$ of the special fibre
there exists a dense open set $U_Y \subset Y$ such that
the restriction of $\xi$ to $U_Y$ is trivial.

We may assume that each $U_Y$ meets no component but $Y$.
The complement of the union of $U_Y$'s  in the special fibre is a finite set $S$ of
points.

By Proposition 6.6 of \cite{HH}, there exists a finite $A$-morphism
 $f : \X \to  \bbP^1_{A}$ with  $S \subset f^{-1}(\infty_{k})$.
 One now replaces the family of $U_{Y}$'s by the  family $\cal{U}$ of irreducible components
 of $f^{-1}(\bbA^1_{k})$. This replaces each $U_{Y}$ by a nonempty affine open set of $U_{Y}$ and one
 replaces $S$ by $f^{-1}(\infty_{k})$.  

By (ii), for each closed point $P$ of the special fibre, for $E$ as in (i) we have
$E(k(P))\neq \emptyset$. Since $E/\X$
is smooth, this implies $E(\hat{R}_{P}) \neq \emptyset$, hence
$E(F_{P}) \neq \emptyset$ (notation as  in  Theorem \ref{localglobalquadratic} and as in \cite{HH}).

 For each open set $U=U_{Y}$, the restriction of $E$ over $U$ is a trivial $G_{U}$-torsor.
 Since $E/\X$ is smooth, this implies   $E(\hat{R}_{U}) \neq \emptyset$ hence $E(F_{U}) \neq \emptyset$.

 Since $E\times_{\X}F$ is a principal homogeneous space under the $F$-algebraic group
 $G\times_{A}F$, for any field extension $L$ of $F$,  the group $G(L)$ acts transitively on $E(L)$.
  An application of  Theorem \cite[Theorem 3.7 and Corollary 3.8] {HHK}   now yields  $E(F) \neq \emptyset$, i.e. $\xi =1 \in H^1(F,G)$.
\end{proof}

\begin{thm}\label{generalizedGroth}
Let $A$ be a complete discrete valuation ring, $K$ its field of fractions
and $k$ its  
residue field.  
Let $\X/A$ be a projective, flat curve over ${\rm Spec} \ A$. 
Assume that $\X$
is connected and regular.
Let $F$ be the function field of $X$.
Let $\Omega$ be the set of all discrete valuations on $F$.

(i) Let $G/A$ be a  (fibrewise connected) reductive group.
If  
 there exists a connected linear algebraic group $H/F$
such that the $F$-group $ (G\times_{A}F)\times_{F} H $ is an $F$-rational variety,
 then
the restriction map with respect to completions
$ H^1(F,G) \to \prod_{v \in \Omega} H^1(F_{v},G)$
has a trivial kernel.

(ii) The restriction map of Brauer groups
${\rm Br} \ F \to \prod_{v \in \Omega} {\rm Br}Ê\ F_{v}$ has a trivial kernel.
 \end{thm}
\begin{proof}
Statement (i) immediately follows from the previous theorem.
As for statement (ii), it  follows from (i) applied to the
 projective linear groups ${\rm PGL}_{n}$.
\end{proof}

   \begin{remark}\label{totallysplit} 
 Using  totally split unramified coverings of  models of Tate curves over a $p$-adic field 
 (see \cite{saito}), one sees that Theorem  \ref {generalizedGroth}~(i)
 does  not   in general  hold  for nonconnected groups, for example for $G={\mathbb Z}/2$.
 A concrete example is given by the ellliptic curve $E$  with affine equation $y^2=x(1-x)(x-p)$
 over the $p$-adic field ${\mathbb Q}_{p}$ ($p$ odd). The rational function $1-x$
 is not a square in the function field  $F={\mathbb Q}_{p}(E)$, but it becomes a square in each completion $F_{v}$
 of $F$. This example is discussed in the appendix to this paper (Section \ref{Appendix}).
 This implies that the patching results of \cite{HHK} in general do not extend to nonconnected groups.
 \end{remark}

 \begin{lem}\label{Agroups} 
 Let $A$ be the ring of integers of a $p$-adic field $K$,
 let $k$ be its
residue field. Let $G/A$ be a  (fibrewise connected) reductive group. 
Then there exists a connected linear algebraic group $H/K$ such that
the $K$-group $(G\times_{A}K) \times_{K} H$ is $K$-rational.
 \end{lem}
 
 \begin{proof}
Let $Z/A$ be the $A$-scheme of Borel subgroups of $G$.
This is a proper and smooth scheme over $\Spec A$.
The special fibre $G_{0}=G\times_{A}k$ is a connected reductive group over the finite field $k$. 
Any such $k$-group is quasisplit (\cite[Chap. III, \S 2.2, Thm. 1]{SCG}). 
Thus $Z(k)\neq \emptyset$, hence $Z(A)\neq \emptyset$ by Hensel's lemma.
There thus exists a Borel  $A$-subgroup   $B \subset G$. 
Let  $T\subset B$ be
its maximal $A$-torus.
Over $K$, the $K$-group $G_{K}=G\times_{A}K$ contains the open set $U^+  \times_{K} U^-\times_{K} (T\times_{A}K)$,
where $U^+ \subset B_{K}$ is the unipotent radical of  
$B_{K}$ 
and $U^-$
is the unipotent radical of the opposite $K$-Borel subgroup of $B_{K} \subset G_{K}$.
Each of these unipotent radicals is $K$-isomorphic to an affine space over $K$.

Let the $k$-torus $T_{0}=T\times_{A}k$ be split by a Galois field extension $k'/k$.
There exists an 
 exact sequence of $k$-tori split by $k'/k$
$$1 \to Q_{0} \to P_{0} \to T_{0} \to 1,$$
where $P_{0}$ is a quasitrivial $k$-torus and $Q_{0}$ is a flasque $k$-torus
(Endo and Miyata, Voskresenski\u{\i}, cf. ~Ê\cite[\S 1, \S 5]{RET}, \cite[\S 0]{CTSflasque}).

Because $k$ is a finite field,  the field extension $k'$ of $k$ is cyclic. By a
theorem of Endo and Miyata (see \cite[Proof of Cor. 3 p.~200]{RET}), 
for any flasque $k$-torus $Q_{0}$ split by
a cyclic extension $k'$ of $k$, there exists  a
 $k$-torus $Q_{1}$ split by $k'$  such that $Q_{0} \times_{k}Q_{1}$ is $k$-isomorphic to a quasitrivial $k$-torus.
If we let $K'/K$ be the cyclic, unramified extension corresponding to $k'/k$, and  we let $A'/A$ 
be the finite, connected, \'etale
Galois cover given by the integral closure  of $A \subset K$ in $K'$,
the sequence of characters associated to the above exact sequence
enables us to produce a sequence of $A$-tori split by $A'/A$, 
$$1 \to Q \to P \to T \to 1,$$
hence in particular a sequence of $K$-tori split by $K'/K$
$$1 \to Q_{K} \to P_{K}  \to T_{K } \to 1,$$
with $Q_{K}$
 a direct factor of a quasitrivial $K$-torus and $P_{K}$ a  quasitrivial $K$-torus
(for basic facts on tori over arbitrary bases, including quasitrivial and flasque tori,
see  \cite[\S 0 and \S 1]{CTSflasque}). 

Because $Q_{K}$ is direct factor of a quasitrivial $K$-torus, Hilbert's theorem 90
implies that the projection $P_{K} \to T_{K}$ has a rational section, hence $Q_{K} \times_{K}T_{K}$
is $K$-birational to $P_{K}$, which is a $K$-rational variety.
Now the product $G_{K} \times_{K} Q_{K}$ is $K$-birational to $U^+  \times_{K} U^-\times_{K} P_{K}$,
which is a $K$-rational variety.
 \end{proof}

\begin{thm}\label{maincor}  Let $A$ be the ring of integers of a $p$-adic field $K$,
 let $k$ be its
residue field.
Let $\X/A$ be a projective, flat curve over $A$. Assume that $\X$
is connected and regular, and that the (reduced) components of the special fibre
are regular.
Let $F$ be the function field of $\X$.  
Let $G/A$ be a  (fibrewise connected) reductive group.  
 \smallskip

If a  class
in $H^1(F,G)$ is unramified at points of codimension 1 on  $\X$, then it is trivial.
\end{thm}

\begin{proof}

 By Theorem \ref{mainprop} (i), there exists $\xi \in H^1(\X,G)$ which restricts to the given class in $H^1(F,G)$.
By hypothesis, each component $Y$ of the special fibre is a regular, hence smooth, projective curve over 
the finite field $k$.
Let us show that the hypothesis of Theorem \ref{mainprop} (ii) is fulfilled.
It is enough to show that for
a smooth, projective, connected curve $Y/k$  
and $G_{k}$ a connected reductive group the image of $H^1(Y,G_{k})$ in $H^1(k(Y),G_{k})$
is trivial.
There exists a central extension
 of  algebraic $k$-groups
$$1 \to Q \to G_{k}^{sc} \times P  \to G_{k} \to 1,$$
where $G_{k}^{sc}$ is a  simply connected semisimple $k$-group, $P$
is a quasitrivial $k$-torus  and
$Q$ is a flasque $k$-torus (\cite[Prop. 3.1]{CTCrelle08}).
 As recalled in the proof of Lemma \ref{Agroups}, because $k$ is finite there exists
a  $k$-torus $Q_{1}$ such that $Q \times_{k} Q_{1}$ is a quasitrivial $k$-torus.
The Brauer group $H^2(Y,{\mathbb G}_{m})$ of a smooth projective curve $Y$ over a finite field is zero.
Since this holds over any finite extension of $k$, this implies
  $H^2(Y,T)=0$ for any   quasitrivial $k$-torus,  
  hence
for any $k$-torus $T$ which is a direct factor of a quasitrivial $k$-torus. Thus $H^2(Y,Q)=0$.

In   the  commutative diagram of exact sequences of pointed \'etale cohomology sets
$$
\begin{CD}
&&H^1(Y, G_{k}^{sc} \times P )   @>>>  H^1(Y,G_{k}) @>>> H^2(Y,Q)  \\
&& @V{}VV       @V{ }VV      @V{}VV 	\\
 && H^1(k(Y), G_{k}^{sc}  \times P) @>>>  H^1(k(Y),G_{k}) @>>>    H^2(k(Y),Q).
  \end{CD}$$
we have $ H^1(k(Y), G_{k}^{sc} )=1$ (Harder \cite{H1,H2}) and $H^1(k(Y),P)=0$
 (Hilbert's theorem 90), and we have proved $H^2(Y,Q)=0$. Thus
 the image of  $H^1(Y,G_{k})$ in $H^1(k(Y),G_{k})$
is trivial.

The claimed result now follows from
  Lemma \ref{Agroups} and Theorem \ref{mainprop} (iii).
 \end{proof}

 \begin{remark}  
Applying Theorem \ref{maincor} to the projective linear groups ${\rm PGL}_{n}$,
 one recovers a proof of Grothendieck's theorem that the 
 Brauer group
 of  a regular proper model $\X/A $ is trivial.  That
 in its turn is closely related
 to the statement that an  element of the Brauer group of $X$ which vanishes at
 each closed point of $X$ is trivial (Lichtenbaum \cite{Licht2}). 
 \end{remark}

 \begin{thm}\label{mainthm}  
 Let $A$ be the ring of integers of a $p$-adic field $K$,
 let $k$ be its
residue field.
Let $X/K$  be a
smooth, projective, geometrically integral curve.
Let $F$ be the function field of $X$.
Let $\Omega$ be the set of all discrete valuations on $F$.
Let $G/A$ be a  (fibrewise connected) reductive group. 
The restriction map with respect to completions
$ H^1(F,G) \to \prod_{v \in \Omega} H^1(F_{v},G)$
has a trivial kernel.
 \end{thm}
\begin{proof}
 One knows
 that  $X/K$ admits a model $\X/A$ as in Theorem \ref{maincor}.
 Let $\xi \in H^1(F,G)$ be in the kernel of the above restriction map.
By   Lemma \ref{cartesiandensity}, the class $\xi$ is unramified at points of
codimension 1 on $\X$.  We conclude by an application of Theorem \ref{maincor}. 
\end{proof} 
 
\begin{remark}

For any integer $n$ and the $A$-group $G={\rm PGL}_{n}$, in the above theorem one may replace $\Omega$ by the set $\Omega_{F/K}$
of discrete valuations on $F$ which are trivial on $K$ : this is just a reinterpretation of
Lichtenbaum's theorem \cite{Licht2}. That this is not so for arbitrary $G$ is shown by the following example.

Let $p$ be an odd prime and $K=\bbQ_{p}$. Let $u$ be a unit in $\bbQ_{p}$ which is not a square.
 Let $X/K$ be the elliptic curve $y^2=x(x+1)(x-p)$. Let $F=K(X)$.  For $a \in F^*$, let
 $(a) \in F^*/F^{*2}=H^1(F,\bbZ /2)$.  Since the divisor of $x \in F^*$ on $X$
  is divisible by 2, the cup-product 
 $\alpha=(x)\cup (u) \cup (p) \in H^3(F,\bbZ /2)$ is unramified at places $v$ of $F$ trivial on $K$, hence
 is trivial in the completion $F_v$ at such a place. The prime $p$ defines a place on $F$,
 the residue field is the function field $\bbF_{p}(Y)$, where $Y$ is the curve defined by
 $y^2=x^2(x+1)$ over $\bbF_{p}$, which is birational to the curve $z^2=x+1$.The residue of $\alpha$ at that place is 
 $(z^2-1) \cup (u) \in H^2(\bbF_{p}(Y),\bbZ/2)$, and  this class is nonzero, since it has a nontrivial residue at $z=1$.
 
This implies : for $G$ the split group of type ${\rm G}_{2}$, 
the restriction map $$ H^1(F,G) \to \prod_{v \in \Omega_{F/K}} H^1(F_{v},G)$$ has a nontrivial kernel.

Lichtenbaum's theorem also  implies that for any    central simple algebra over $D$ over $K$,
and $G$ the $F$-group ${\rm PGL}_{D}$, the above map has a trivial kernel. The above example
shows that this is not so for  the $K$-group  $G={\rm SL}_{D}$,
 where $D$ is the quaternion algebra $(u,p)$ over $K=\bbQ_{p}$.
 \end{remark}

 \begin{remark}  
 Let $A$ be the ring of integers of a $p$-adic field $K$. Let $G/A$ be a (connected) reductive group.
Let $F=K(X)$ be the function field
 of a smooth geometrically integral curve over $K$.  Let $\X$ be a regular   model of
 $X$ over $A$.   Assume that  the fibres of $G\to \Spec A$ are  {\it simply connected}
 (this is equivalent to the assumption that the $K$-group  $G\times_{A}K$ is simply connected).
Then for $\xi \in H^1(F,G)$ and $x$ a point  of codimension~1 on $\X$, defining a valuation $v$ on
$F$
with associated completion $F_{v}$,
the conditions

(i) $\xi$ is unramified at $x$ (as in Theorem \ref{maincor})

(ii) $\xi$ has trivial image in $H^1(F_{v},G)$ (as in Theorem \ref{generalizedGroth}~(i)  and Theorem  \ref{mainthm})

\noindent are equivalent.

 Indeed, for any point $x$ of codimension~1 on a regular model $\X$, with
 complete local ring $\hat{A_{x}}$ and residue class field $\kappa(x)$,
 we have $H^1(\hat{A_{x}},G) \simeq H^1(\kappa(x),G)$ (Hensel's lemma)
 and $H^1(\kappa(x),G)=1$ whether $x$ lies on the generic fibre of $\X/A$  (Kneser, Bruhat-Tits)  or $x$ is a generic point of a component of the special fibre 
 of $\X/A$ (Harder \cite{H1,H2}).
  \end{remark}

 \section{Connection to Rost's invariant and a theorem of Kat\^o} \label{RostKato}

For any simply connected, absolutely almost simple semisimple  group
$G$ over
a field $F$ of characteristic zero,
we have Rost's invariant (see \cite[Chapter VII, Section 31]{KMRT}) :
$$R_G:  H^1(F,G) \to H^3(F,\bbQ/\bbZ(2)).$$

In a number of cases, this map has a trivial kernel.
Such is the case if $G=\SL(D)$ for $D/F$ a central simple algebra
of squarefree index (Merkurjev--Suslin).
Such is the case  for quasisplit groups of type
${}\sp {3,6}D\sb 4$ (\cite[40.16]{KMRT}, \cite[Thm. 6.14]{Ch})
or of type $E\sb 6,E\sb 7$ (Garibaldi  \cite[Theorem 0.1]{Garibaldi}, see also \cite[Thm. 6.1]{Ch}).
Such is the case for the split group $G_{2}$ (\cite[Thm.~9]{S}).
Such is the case for the split group $F_{4}$ (\cite[\S 9.4]{S}) 
It is not reasonable to hope for a positive answer for an arbitrary such $G$,
as examples with $G=\Spin(q)$ show.

\medskip

For fields of cohomological dimension at most 2, the triviality of the
kernel of the Rost invariant $R_{G}$  is none other than
Serre's conjecture II for $G$,  which in this generality is still unknown
for $G$   of type $E_{8}$.

\begin{remark}
For fields of cohomological dimension 3 and $G$ arbitrary,
$R_{G}$ may have a nontrivial kernel, as shown by the following
example due to Merkurjev, and which we publish with his kind permission.
There exists a field $k$   of characteristic 0 and 
of cohomological dimension 2  over which there exist
a central simple division algebra $A=H_{1}\otimes_{k} H_{2}$
with $H_{1}$ and $H_{2}$ quaternion algebras (\cite[Teorem 4]{Merk0}).
Let $F$ be either $k(t)$ or $k((t))$.
Then $F$ has cohomological dimension3. 
The reduced norm  of $A$ is a homogeneous form of degree 4
without a zero over $k$. Thus
$t^2 \in F$  is not a reduced norm of $A\otimes_{k}F$.
That is, the class of $t^2$ in $F^*/{\rm Nrd}(A_{F}^*)=H^1(F,G)$,
with $G=SL_{1}(A)$,
is nontrivial. Let $[A] \in H^2(k,\mu_{4}) \subset \Br k$ be the class of $A$.
By \cite[p. 437]{KMRT} (for more details, see \cite[p. 138]{Merk}),
 the Rost invariant $R_{G}$ sends $t^2 \in H^1(F,G_{F})$
to the cup-product $t^2 \cup [A] \in H^3(F,\mu_{4}^{\otimes 2}) \subset
H^3(F,\bbQ/\bbZ(2))$ (here $t^2$ is taken in $F^*/F^{*4}=H^1(F,\mu_{4})$).
Since $2[A]=0 \in \Br  k$, this cup-product is zero.
\end{remark}

When $G$ is quasisplit, not of type $E_8$,  the situation is much better.
 The following proposition
is certainly known to specialists.

\begin{prop}\label{rostspin}
Let $F$ be a field of characteristic not $2$ and of $2$-cohomological
dimension at most $3$. Let $q_0$ be a quadratic form over $F$ which
is isotropic and of dimension at least $5$.
Let $G = \Spin(q_0)$. Then the kernel of the Rost map
$H^1(F,G) \to H^3(F,\bbQ/\bbZ(2))$ is trivial.
\end{prop}

\begin{proof}

Let
$$1 \to \mu_{2 }  \to \Spin(q_{0})  \to \SO(q_{0}) \to 1$$
be the central isogeny from the Spin group to the special orthogonal group.
This gives rise to an exact sequence of pointed Galois cohomology sets
$$ \SO(q_{0})(F)\stackrel{\delta _{0} }{\to} H^1(F,\mu_{2})
\stackrel{i}{\to} H^1(F,\Spin(q_{0}) ) \stackrel{j}{\to}H^1(F,
\SO(q_{0})).$$
For $\xi \in H^1(F,\Spin(q_{0}) )$, the class $j(\xi)$ corresponds to  
a quadratic form $q_{1}$ with dimension $dim(q_{0})=dim(q_{1})$,
discriminant $disc(q_{0})=disc(q_{1})$ and
Clifford invariant $c(q_{0})=c(q_{1})$. Then in the Witt group
$W(F)$  the class $q_{1} \perp -q_{0}$
lies in the third power $I^3(F)$ of the fundamental ideal
and  its Arason invariant $e_{3}(q_{1} \perp -q_{0}) \in
H^3(F,\mu_{2})$, which coincides
with the Rost invariant of $\xi$ (\cite[Page 437]{KMRT}),
is zero. Now the hypothesis $cd_2(F) \leq 3$ implies that
$ H^4(F,\mu_{2})=0$, $I^4(F)=0$
and that $e_{3} : I^3(F) \to H^3(F,\mu_{2})$
is an isomorphism (\cite{MS}, \cite{Rost}, \cite[Cor. 4, Thm. 2]{AEJ}).
The two forms $q_{0}$ and $q_{1}$ have the same dimension.
By Witt simplification  they are isomorphic. Thus $j(\xi)=1$
hence $\xi=i(\eta)$ for some $\eta \in H^1(F,\mu_{2})$. Since $q_{0}$
is isotropic,
the connecting map $\delta _{0} : \SO(q_{0}) \to
H^1(F,\mu_{2})=F^*/F^{*2}$, which is
the spinor map, is onto.
Thus $\xi=1 \in H^1(F,\Spin(q_{0}) )$.

\end{proof}

\begin{thm}\label{rostinjectifcd3}
Let $F$ be a field of characteristic zero and of cohomological
dimension at most 3. Let $G/F$ be an absolutely almost simple, simply
connected, quasisplit  semisimple group. Assume that $G$
is not of type $E_{8}$.
Then  the kernel of the Rost map $H^1(F,G) \to H^3(F,\bbQ/\bbZ(2))$
is trivial.
\end{thm}

\begin{proof}
The cases $^1A\sb n$ and $C\sb n$ are trivial, since in these cases
$H^1(F,G)=1$
over any field $F$. For quasisplit groups of type
${}^{3,4}D\sb 4,E\sb 6,E\sb 7$, $G_{2}$ and $F_{4}$ the kernel
is trivial over any field $F$ of characteristic zero (see references above).

Let $G$ be of type ${}^{2}A\sb n$, quasisplit, $n \geq 2$. There is
a quadratic extension $L/F$ and an $L/F$--hermitian form $h$ of dimension
$n+1$ such that $G = \SU(h)$. Further $G$ quasisplit implies that  $G$
is isotropic (\cite[20.6 (ii), p.~225]{Borel}), which in turn implies that the hermitian form $h$ is
isotropic (\cite[23.8, p.~264]{Borel}).
 Let $V$ be the underlying space of $h$. Then the map
$q : V \to F$ given by $q(v)=h(v,v)$ is a quadratic form of dimension $2n+2$
over $F$ which is isotropic. Further there is a homomorphism
$\alpha: \SU(h) \to \Spin(q)$ such that the composite map
$$ H^1(F, \SU(h) \to H^1(F, \Spin(q) \to H^3(F,\bbQ/\bbZ(2)),$$
is the Rost invariant for $\SU(h)$ where the first map is
induced by $\alpha$ and the second one is the Rost invariant
for $\Spin(q)$ (\cite[31.44, p.~438]{KMRT}). The  triviality of the kernel of the
Rost invariant in this case follows from
  Proposition  \ref{rostspin}.

Let $G$ be of type $B_n, n\geq 2$, or $^1D_n$  or $^2D_n$, $n\geq 3$, which is
quasisplit. Then $G$ is isomorphic to $\Spin(q)$ for some quadratic
form $q$ over $F$ of dimension at least 5; further, $G$ quasisplit
implies that  $G$
is isotropic, which in turn implies
 that the quadratic form $q$ is isotropic (\cite[ 23.4, p.~256]{Borel}).
In this case the triviality of the kernel of the Rost invariant follows from
Proposition  \ref{rostspin}.

This completes the proof of the  triviality of the kernel  
of the Rost invariant for  all quasisplit groups not of type $E_8$.
\end{proof}

 \bigskip

By combining Theorem \ref{rostinjectifcd3} and a theorem of Kat\^o,
one gets a proof of Conjecture 2 of the introduction for quasisplit groups without $E_{8}$-factors.
That proof is independent of the other sections of the present paper.

 \begin{thm}\label{rostkato}
 Let   $K$ be a $p$-adic field.
Let $X/K$ be a smooth, projective,
geometrically integral curve.  Let $F=K(X)$  be the function field of $X$.
Let $\Omega$ denote the set of discrete valuations 
on the field $F$. Given $v \in \Omega$ we let $F_{v}$ denote the completion of
$F$ at $v$.
Let $G/F$ be a quasisplit,   simply connected group, absolutely almost simple group 
without $E_{8}$ factor.
Then the
  kernel of the     diagonal map $H^1(F,G) \to \prod_{v \in \Omega}
 H^1(F_{v},G)$
 is trivial.
 \end{thm}

 \begin{proof}
 The field $F=K(X)$ is of cohomological dimension 3.
 The result immediately follows from  the combination of   Theorems \ref{rostinjectifcd3}  
 and  a theorem of  Kat\^o \cite{Kato}: 
 For $X/K$ as in the statement of the theorem, the kernel of the diagonal restriction map
 $$H^3(F,\bbQ/\bbZ(2)) \to \prod_{v \in \Omega} H^3(F_{v},\bbQ/\bbZ(2)) $$
 is trivial (here it is enough to consider the $v$'s associated to the codimension 1 
 points on a regular proper model of $X$ over the ring of integers of $K$).
 \end{proof}
 
 The hypotheses of the above theorem should be compared with  those of Theorem  \ref{mainthm},
 whose proof builds upon the work of Harbater, Hartmann and Krashen.

  Using  Theorem \ref{mainthm}  together with Bruhat-Tits theory,
  we now show that Theorem \ref{rostinjectifcd3}
  also holds for    groups of type $E_{8}$ over $F(X)$.

\begin{thm}\label{rostinjectifK(X)}
Let $A$ be the ring of integers of a $p$-adic field $K$.
Let $X/K$ be a smooth, projective,
geometrically integral curve.  Let $F=K(X)$  be the function field of $X$.
 Let $G$ be an absolutely  almost simple,  
  simply connected  semisimple group over $A$.
If $G$ is of type $E_{8}$, assume that the residue characteristic
is different from 2, 3 and 5.
Then the kernel of the Rost map $ H^1(F,G) \to H^3(F,\bbQ/\bbZ(2))$
is trivial.
\end{thm}
 
 \begin{proof}
As explained in the proof of Lemma \ref{Agroups}, the group $G/A$ is
 automatically quasisplit.

 We may restrict the set of places under consideration to the set of
 points of codimension 1 on a regular proper model ${\cal X}/A$.
 Let $O_{v}$ be the ring of integers in $F_{v}$.
  The residue field $\kappa=\kappa_{v}$ at such a place is either a $p$-adic field
 or a function field in one variable over a finite field.

 Let $F_{v}^{nr}$ be the maximal unramified extension of $F_{v}$.
 At a  prime $l$ different from the
 residue characteristic, the $l$-cohomological
 dimension of $F_{v}^{nr}$ is 1, at the residue characteristic  
it  is at most 2 (\cite[chap. II, \S 4.3, Prop. 12 p. 95]{SCG}).

 From this we deduce $H^1(F_{v}^{nr},G)=1$ 
 if the residue characteristic is not a torsion prime of $G$;
this would hold even if the connected group $G$ was not simply connected, this follows
 from Steinberg's theorem, see \cite[Thm. 4'']{S}.

 In the general case, i.e. when the residue characteristic
 is a torsion prime and the quasisplit group $G$ is simply connected and not of type $E_{8}$, 
 we resort to the known (case by case) theorem that for such a group
 over a field $L$ of characteristic zero and of cohomological dimension 2, 
 we know $H^1(L,G)=1$ (Merkurjev and Suslin, Bayer-Fluckiger and Parimala, Gille,
 Chernousov, see \cite{S} and \cite{GilleconjII}).
Under our hypotheses, we thus have
   $H^1(F_{v}^{nr}/F_{v},G)  \oii H^1(F_{v},G)$.
   
 By Theorem \ref{mainthm}, to prove the theorem
  it is enough to show that for any $v$
as above  the kernel of the map
 $ H^1(F_{v},G) \to H^3(F_{v},\bbQ/\bbZ(2))$
 is trivial.

 Let $\xi \in H^1(F_{v}^{nr}/F_{v},G)$.
 Since $G/A$ is a reductive $A$-group,
 Bruhat-Tits theory, 
 as developed in Gille's paper \cite[Thm. 3', Thm. 4']{Gille}
  shows that there exist  a parahoric subgroup $P$ over $O_{v}$ 
 and a class $\eta \in H^1(O_{v},P)$ with the following properties.
 
 The image of $\eta$ under the natural map $ H^1(O_{v},P) \to  H^1(F_{v}^{nr}/F_{v},G)$
 is $\xi$.

The reductive quotient
 $M/\kappa$ of $P\times_{O_{v}}\kappa$ 
 is part of an exact sequence of  reductive groups   over the field $\kappa$:
 $$1 \to \bbG_{m} \to M' \to M \to 1,$$
 where 
  $M'/\kappa$  is the product of
 a simply connected semisimple group and a torus which is a direct factor
 of a quasitrivial torus.

 Let us consider the composite map
 $$ H^1(F_{v}^{nr}/F_{v},G)   \to {\rm Ker} [H^3(F_{v},\bbQ/\bbZ(2)) \to H^3(F_{v}^{nr},\bbQ/\bbZ(2))]
  \to H^2(\kappa,\bbQ/\bbZ(1)) \oii {\rm Br} (\kappa),$$
 where the map from $H^3$ to $H^2$ is the usual residue map for primes
different from the residue characteristic, and the Kat\^o residue map in general.
  The image of $\xi$ under this composite map  coincides with the image of $\eta$ under the composite map
 $$ H^1(O_{v},P) \oii  H^1(\kappa,P\times_{O_{v}} \kappa) \oii  H^1(\kappa,M) \to H^2(\kappa,\bbG_{m})={\rm Br}(\kappa),$$
 where the map $H^1(\kappa,M) \to H^2(\kappa,\bbG_{m})$ is the boundary map from the above sequence.

Since $\kappa$ is either a $p$-adic field or a function field in one variable over a finite field, 
 $H^1(\kappa,M')=0$. Hence the map $H^1(\kappa,M)  \to {\rm Br}(\kappa)$ has trivial kernel.
 
 If the image of $\xi \in H^1(F_{v},G) = H^1(F_{v}^{nr}/F_{v},G)$ in $H^3(F_{v},\bbQ/\bbZ(2))$ is zero,
 we conclude that $\eta=1$ hence that $\xi=1$.
 \end{proof}

 \section {Appendix}\label{Appendix}

 In this appendix we present a down-to-earth computation   for the phenomenon mentioned
in Remarks \ref{notwithtwo} and \ref{totallysplit}.
Let $p$ be an odd prime.
Let $E$ be the elliptic curve over $\bbQ_p$ defined by the affine equation
$$y^2=x(1-x)(x-p).  \hskip3mm ({\rm A}1)$$
Let $F=\bbQ_{p}(E)$ be its function field.
We clearly have
$$ x+(1-x)=1 ; \hskip1mm x-(x-p)=p ;  \hskip1mm (1-x) + (x-p)=1-p. \hskip3mm ({\rm A}2)$$
Let $A \subset F$ be a proper discrete valuation ring, let
$v$ denote the valuation, $A_{v}$ the completion of $A$ and $F_{v}$
the fraction field of $A_{v} $, i.e. the completion of $F$ at $v$.
Let $k$ be the residue field of $A$.
Let $\pi$ be uniformizing parameter for $A$.
If $v(p)=0$, the field $k$ is a finite extension of $\bbQ_{p}$.
If $v(p) \neq 0$ then $v$ induces on $\bbQ_{p}$ a proper valuation
and $v(p)>0$.
In both cases, $1-p$ is a square in $k$ hence is  a unit which is a square in
$A_{v}$.

\medskip

Claim : {\it The function $1-x$ is not a square in $F$, but it is a square in
each completion $F_v$.}

\medskip

That $(1-x)$ is not a  square is proved by considering the quadratic
extension $F/\bbQ_{p}(x)$ : the kernel of the restriction map on square classes is 
$\bbZ/2$, spanned by the class of $x(1-x)(x-p)$.

Assume $v(1-x) <0$. Then $v(x)=v(1-x)=v(x-p)$ hence (A1) gives that each of these
is even. Let $x=u/\pi^{2n}$ with $u \in A^*$ and $n>0$. From (A1) we get
that $-u^3$, hence $-u$ is a square in $A_{v}$. Now $1-x=( \pi^{2n}-u)/\pi^{2n}$
is a square in $K_{v}$.
Assume $v(1-x)>0$. Then $v(x)=0$ and $v(x-p)=0$. From (A1) we get
$1-x=u\pi^{2n}$ with $u$ a unit and $n>0$. Then from (A2)
we get that $x$ and $x-p$ are squares in $A_{v}$. But then (A1) shows that
$1-x$ is a square in $K_{v}$.
Now assume $v(1-x)=0$.  If $v(x)>0$ or $v(x-p)>0$, then (A2) implies
that $1-x$ is a square in $A_{v}$.
We are reduced to the case where $v(1-x)=v(x)=v(x-p)=0$.
If $v(p)>0$, then $x(x-p)$ is a square in $A_{v}$. From (A1) we deduce
that $1-x$ is a square in $A_{v}$.

We are reduced to the case $v(1-x)=v(x)=v(x-p)=0$, hence $v(y)=0$,
 and $v(p)=0$.
That is, the valuation $v$ corresponds to a closed point  $M$ on the
elliptic curve $E$ over $\bbQ_{p}$ (the trivial valuation is excluded).
The ring $A$ is the local ring of $E$ at that point.
The point $M$ lies on the affine curve
$y^2=x(1-x)(x-p).$
Let $k/\bbQ_{p}$ be the corresponding finite  field extension.
Thus $k$ is the residue field of $A_{v}$. 
Let $B$ be the ring of integers of $k$.
The reductions of $x$ and $y$ modulo the maximal ideal of $A$
give rise to elements  $a,b \in k$ with $b(1-b)(b-p) \neq 0 $
and
$a^2=b(1-b)(b-p) \neq 0$.
The element $1-x$ is a square in $A_{v}$ if and only if $1-b$
is a square in $k$.
To show that this is indeed the case,
we do exactly  the same computations in $k$, with respect to the valuation
$w$ of $k$, which satisfies $w(p)>0$,
as we had done in $F$. The computation is identical, it stops at the analogue of
the end of the previous paragraph.

\bigskip

{\it Acknowledgements} We thank D. Harbater for his remarks on the typescript.
 
 Work for this paper was partially supported by Emory University (J.-L. Colliot-Th\'el\`ene and 
V. Suresh), NSF grant DMS-0653382 (R. Parimala) and Max Planck Institut f\"ur
  Mathematik, 
Bonn (JLCT and RP).


\vskip1cm

\begin{thebibliography}{AOVdB}


 


\bibitem{AEJ} J. K. Arason, R. Elman and B. Jacob, Fields of cohomological $2$-dimension
three, Math Ann. {\bf 274} (1986), 649--657.

\bibitem{Borel} A. Borel,  {\it Linear algebraic groups},
Second, enlarged edition. Graduate Texts in Mathematics {\bf 126}. Springer-Verlag, New York, 1991.

\bibitem{Brumer} A. Brumer,   Remarques sur les couples de formes quadratiques,   C. R. Acad. Sci. Paris S\'er. A-B  286  (1978), no. 16, A679--A681.
 

\bibitem{Ch} V. Chernousov, The kernel of the Rost invariant, Serre's conjecture II and the Hasse principle for quasi-split groups ${}\sp {3,6}D\sb 4,E\sb 6,E\sb 7$,  Math. Ann.  {\bf 326}  (2003),  no. 2, 297--330.


\bibitem{CTCrelle08}  J.-L. Colliot-Th\'el\`ene, R\'esolutions flasques des groupes lin\'eaires connexes,
J. reine angew. Math. {\bf 618} (2008) 77--133.


 \bibitem{CT}  J.-L. Colliot-Th\'el\`ene, Fibres sp\'eciales des hypersurfaces de petit degr\'e, C. R. Acad. Sc. Paris {\bf 346}   (2008)  63--65.
 
 \bibitem{RET}{ } J.-L. Colliot-Th\'el\`ene et J.-J.  Sansuc,  La R-\'equivalence sur les tores,
 Ann. Sci. \'Ecole Norm. Sup. (4) {\bf 10} (1977) 175--229.


\bibitem{CTSa}{ } J.-L. Colliot-Th\'el\`ene et J.-J.  Sansuc,  
  Fibr\'es quadratiques et composantes connexes r\'eelles,   Math. Ann.  {\bf 244}  (1979), no. 2, 105--134.
  
  \bibitem{CTSflasque} J.-L. Colliot-Th\'el\`ene et J.-J.  Sansuc,  Principal homogeneous spaces
  under flasque tori: applications, J. Algebra {\bf 106} (1987) 148--205.
  
  
   \bibitem{Garibaldi}{}
 R. S. Garibaldi, 
The Rost invariant has trivial kernel for quasi-split groups of low rank,
Comment. Math. Helv. {\bf 76} (2001), no. 4, 684--711. 
 
   
 \bibitem{Gille}{} P. Gille, Invariants cohomologiques de Rost en caract\'eristique positive,
 K-Theory {\bf 21} (2000) 57--100.
 
 \bibitem{GilleconjII}{} P. Gille, Cohomologie galoisienne des groupes quasi-d\'eploy\'es
  sur des corps de dimension cohomologique $\leq 2$, Compositio math. {\bf 125} (2001) 283--325.



\bibitem{Giraud}{}
J. Giraud, {\it Cohomologie non ab\'elienne}, Grundlehren der mathematischen Wissenschaften,
Bd. {\bf 179} Berlin, Heidelberg, New York, Springer 1971.

\bibitem{HH}{ }    D. Harbater and J. Hartmann, Patching over Fields,  
arXiv:0710.1392v3 [math. AG]    September  27,  2008, to appear in Israel Journal of Mathematics.

\bibitem{HHK}{ } D. Harbater, J. Hartmann and D. Krashen, 
Applications of patching to quadratic forms and central simple algebras,
 Invent. math.
   {\bf 178} (2009) 231--263.
 

\bibitem{H0}{ }
G. Harder, Halbeinfache Gruppenschemata \"uber Dedekindringen,
Inventiones math. {\bf 4} (1967) 165--191.

\bibitem{H1}{ }
G. Harder,  
Chevalley groups over function fields and automorphic forms,
Ann. of Math. (2) {\bf 100} (1974), 249--306. 


\bibitem{H2}{ }
G. Harder,  
\"Uber die Galoiskohomologie halbeinfacher algebraischer Gruppen. III,
Collection of articles dedicated to Helmut Hasse on his seventy-fifth birthday, III.
J. reine angew. Math. {\bf 274/275} (1975), 125--138.


\bibitem{JT}{} B.  Jacob and J.-P. Tignol, appendix to 
D. J. Saltman, Division algebras over p-adic curves, J. Ramanujam Math. Soc. {\bf 12}
(1997) no. 1, 25--47. Correction, ibid. {\bf 13} (1998) no. 2, 125-129.


\bibitem{Kato}{ } K. Kat\^o,
 A Hasse principle for two-dimensional global fields,   J. reine angew. Math.  {\bf 366}  (1986), 142--181.
 
 \bibitem{KK} K. Kat\^o and T. Kuzumaki,
 The dimension of fields and algebraic $K$-theory,
J. Number Theory {\bf 24}  (1986), no. 2, 229--244. 


 \bibitem{KMRT} M-A. Knus, A. Merkurjev, M. Rost, J.-P. Tignol,  {\it The book of involutions}. 
With a preface in French by J. Tits. American Mathematical Society Colloquium Publications, {\bf 44}. 
 American Mathematical Society, Providence, RI, 1998. 
 
\bibitem{Lam} T. Y. Lam, {\it The algebraic theory of quadratic forms}, Benjamin 1973.

\bibitem{Licht1} S. Lichtenbaum, Curves over discrete valuation rings, 
Amer. J. Math. {\bf 25} no. 2 (1968)
380--405. 
 

\bibitem{Licht2} S. Lichtenbaum,  
Duality theorems for curves over $p$-adic fields,
Invent. math. {\bf 7} 1969 120--136. 


 \bibitem{Merk}  A.  S. Merkurjev,   Rost invariants of simply connected algebraic groups, with a section by Skip Garibaldi,  
 in  {\it Cohomological invariants in Galois cohomology}, Univ. Lecture Series   {\bf 28},   101--158, Amer. Math. Soc., Providence, RI, 2003.  
 
 \bibitem{Merk0} Merkurjev, A. S.
Simple algebras and quadratic forms,   Izv. Akad. Nauk SSSR Ser. Mat. 55 (1991), no.1, 218--224; transl. Math. USSR-Izv. 38 (1992), no. 1, 215--221.
 
 \bibitem{MS} A. S. Merkurjev and A. A. Suslin, The norm-residue homomorphism of degree 3 (in Russian), Izv. Akad. Nauk SSSR 54 (1990), 339--356; English translation: Math. USSR Izv. 36 (1991), 349--368. 
 




 \bibitem{Nis}{ }
 Ye. A. Nisnevich,  Espaces homog\`enes principaux rationnellement triviaux et arithm\'etique des sch\'emas en groupes r\'eductifs sur les anneaux de Dedekind,   C. R. Acad. Sci. Paris S\'er. I Math.  {\bf 299}  (1984),  no. 1, 5--8.
 

\bibitem{PS1}{ } R. Parimala and V. Suresh, Isotropy of quadratic forms over function fields
in one variable over $p$-adic fields, Publ. math. I.H.\'E.S. {\bf 88} (1998) 129--150.

\bibitem{PS2}{ } R. Parimala and V. Suresh,  The $u$-invariant of the function fields of $p$-adic curves,
arXiv:0708.3128v1 [math NT] 23 August 2007, to appear in Ann. of Math.



 \bibitem{Rost} M. Rost, On Hilbert Satz 90 for $K_{3}$ for degree-two extensions, preprint, Regensburg, 1986; \url{http://www.mathematik.uni-bielefeld.de/~rost/K3-86.html.} 

\bibitem{RST}{ }   L. H.  Rowen, A. S. Sivatski and J.-P.  Tignol,  Division algebras over rational function fields in one variable, in {\it   Algebra and Number Theory},  158--180, Hindustan Book Agency, Delhi, 2005.  

\bibitem{saito} S. Saito, Class field theory for curves over local fields,  J. Number Theory  {\bf 21}  (1985),  no. 1, 44--80.

\bibitem{SCG}{ } J-P. Serre,    {\it Cohomologie galoisienne}, Springer LNM {\bf 5} (1964);
cinqui\`eme  \'edition, r\'evis\'ee et compl\'et\'ee (1994).

\bibitem{S}{ } J-P. Serre,  Cohomologie galoisienne : progr\`es et probl\`emes,
S\'eminaire Bourbaki, Expos\'e  783 (Mars 1994), Ast\'erisque {\bf 227} (1995).
 

\bibitem{Sha}{} I. R. Shafarevich,
{\it Lectures on minimal models and birational transformations of two dimensional schemes}, Notes by C. P. Ramanujam, Tata Institute of Fundamental Research Lectures on Mathematics and Physics {\bf  37}, Tata Institute of Fundamental Research, Bombay 1966.



\end{thebibliography}
\end{document}